\documentclass[12pt, reqno]{amsart}
\usepackage{amsmath, amsthm, amscd, amsfonts, amssymb, graphicx, color}
\usepackage[bookmarksnumbered, colorlinks, plainpages]{hyperref}
\usepackage{float}

\makeatletter \oddsidemargin.9375in \evensidemargin \oddsidemargin
\newtheorem{theorem}{Theorem}[section]
\newtheorem{lemma}[theorem]{Lemma}
\newtheorem{proposition}[theorem]{Proposition}

\theoremstyle{definition}

\newtheorem{example}[theorem]{Example}

\theoremstyle{remark}

\numberwithin{equation}{section}

\begin{document}
\setcounter{page}{1}

%%%%%%%%%%%%%%%%%%%%%%%%%%%%%%%%%%%%%%%%%%%%%%%
%% Please do not remove the following statement.
%%%%%%%%%%%%%%%%%%%%%%%%%%%%%%%%%%%%%%%%%%%%%%%

%%%%%%%%%%%%%%%%%%%%%%%%%%%%%%%%%%%%%%%%%%%%%%%

%%%%%%%%%%%%%%%%%%%%%%%%%%%%%%%%%%%%%%%%%%%%%%%%%%%%%%%%%%%%%%%%%%%%%
% Insert title of your article. Note: \title[short title]{full title}
%%%%%%%%%%%%%%%%%%%%%%%%%%%%%%%%%%%%%%%%%%%%%%%%%%%%%%%%%%%%%%%%%%%%%
\title[GENERALIZED t-EDGE DISTANCE-BALANCED GRAPHS]{GENERALIZED t-EDGE DISTANCE-BALANCED GRAPHS}
%%%%%%%%%%%%%%%%%%%%%%%%%%%%%%%%%%%%%%
% Author's name must be inserted here
%%%%%%%%%%%%%%%%%%%%%%%%%%%%%%%%%%%%%%
\author[Aliannejadi, Alaeiyan, Gilani and Asadpour]{Z. Aliannejadi$^{*}$, M. Alaeiyan, A. Gilani}
%%%%%%%%%%%%%%%%%%%%%%%%
\thanks{{\scriptsize
\hskip -0.4 true cm MSC(2010): Primary: 05C12; Secondary: 05C25,
\newline Keywords: generalized $t$-edge distance-balanced graphs, graph products, edge-szeged index, complete bipartite graphs\\
$*$Corresponding author }}

%%%%%%%%%%%%%%%%%%%%%%%%%%%%%%%%%%%%%%%%%%%
\begin{abstract}
A connected and nonempty graph $A$ is defined as generalized $t$-edge distance-balanced, while for each edge $f=\alpha\beta$ the number of edges nearer to $\alpha$ than $\beta$ are equal to $t$-times of edges nearer to $\beta$ than to $\alpha$, for $t\in\mathbb{N}$, or vice versa. We determine some classes of such graphs. Moreover, we investigate edge-szeged index of the generalized $t$-edge distance-balanced graphs. Also, it is discussed about their cartesian and lexicographic products.
\end{abstract}

%%%%%%%%%%%%%%%%%%%%%%%
\maketitle
%%%%%%%%%%%%%%%%%%%%%%%
\section{Introduction}
The notion of graph is a pivotal tool to make use of the modelling of the phenomena and it is taken into consideration in many studies in a recent decades. One of the optimal uses of graphs theory is to classify graphs based on discriminating quality. This phenomenon can be best observed in distance-balanced graphs has been determined by [13]. Also, it is investigated in some papers, we refer the reader to ([1],[2],[3],[4],[5],[6],[10],[11],[14]-[19]) and references therein.

We consider $A$ is a connected, finite and undirected graph throughout of this paper, in which its vertex set is $V(A)$ and its edge set is $E(A)$. In graph $A$, the interval among vertices $\alpha,\beta\in V(A)$ is introduced the number of edges in the least interval joining them and it is indicated by $d_A(\alpha,\beta)$ (see [4, 18]). For every two desired vertices $\alpha,\beta$ of $V(A)$ we indicate $n_\alpha^A(f)=|W_{a,\beta}^A|=|\{a\in V(A)| d_A(a,\alpha)< d_A(a,\beta)\}|$. In the same way, we would define $n_\beta^A(f)=|W_{\beta,\alpha}^A|$. We name $A$ distance-balanced ($DB$) while for adjacent vertices $\alpha$ and $\beta$ of $A$, we have
$|W_{\alpha,\beta}^A|=|W_{\beta,\alpha}^A|$.\\

For each two desired edges $f=\alpha\beta$, $\acute{f}=\acute{\alpha}\acute{\beta}$, the distance between $f$ and $\acute{f}$ is introduced via:

\begin{center}
$d_A(f,\acute{f})=\min\{d_A(\alpha,\acute{f}), d_A(\beta,\acute{f})\}=\min\{d_A(\alpha,\acute{\alpha}),d_A(\alpha,\acute{\beta}),d_A(\beta,\acute{\alpha}),d_A(\beta,\acute{\beta})\}$.
\end{center}
Set$\hspace*{0.2cm}$ $M_\alpha^A(f)=\{\acute{f}\in E(A)| d_A(\alpha,\acute{f})< d_A(\beta,\acute{f})\}$ and $m_\alpha^A(f)=|M_\alpha^A(f)|$,\\
$\hspace*{0.68cm}$ $M_\beta^A(f)=\{\acute{f}\in E(A)| d_A(\beta,\acute{f})< d_A(\alpha,\acute{f})\}$ and $m_\beta^A(f)=|M_\beta^A(f)|$,\\
and $\hspace*{0.007cm}$ $M_0^A(f)=\{\acute{f}\in E(A)| d_A(\alpha,\acute{f})= d_A(\beta,\acute{f})\}$ and $m_0^A(f)=|M_0^A(f)|$. 

Persume that $f=\alpha\beta\in E(A)$. For every two integers $i,j$ we consider:
\begin{center}
$\acute{D}_j^i(f)=\{\acute{f}\in E(A)|d_A(\acute{f},\alpha)=i, d_A(\acute{f},\beta)=j\}$.\\
\end{center}

A "distance partition" of $E(A)$ is concluded by sets $\acute{D}_j^i(f)$ due to the edge
$f=\alpha\beta$. Merely the sets $\acute{D}_i^{i-1}(f)$, $\acute{D}_i^i(f)$ and, $\acute{D}_{i-1}^i(f)$, for each $(1\leqslant i\leqslant d)$ might be nonempty according to the triangle inequality (The diameter of
the graph $A$ is $d$), as well as $\acute{D}_0^0(f)=\phi$.\\

A graph $A$ is said to be generalized $t$-edge distance-balanced for $t\in\mathbb{N}$ and $t>1$, so that for each edge $f=\alpha\beta$ of $E(A)$, either $m_\alpha^A(f)=t m_\beta^A(f)$ or $m_\beta^A(f)=t m_\alpha^A(f)$. For convenience, we would write $Gt-EDBG$ instead of generalized $t$-edge distance-balanced graph and also $G t-EDB$ instead of generalized $t$-edge distance-balanced. In the same way, we define generalized $t$-distance-balanced graphs ($Gt-DB$), in which $n_\alpha^A(f)=t n_\beta^A(f)$, or $n_\beta^A(f)=t n_\alpha^A(f)$.\\

We also remind that the sets $W_{\alpha,\beta}^A$ and $M_\alpha^A(f)$ become visible as the well-known Szeged index and edge-Szeged index of a graph $A$ in the chemical graph theory, in which $S_z(A)=\sum_{\alpha\beta\in E(A)}|W_{\alpha,\beta}|.|W_{\beta,\alpha}|$ and $S_{{z}_{e}}(A)=\sum_{\alpha\beta\in E(A)}m_{\alpha}^A(f).m_{\beta}^A(f)$, respectively, cf. [7,8,9].\\

A graph $A$ is called \textit{edge distance-balanced} (briefly $EDB$), while we have $m_\alpha^A(f)=m_\beta^A(f)$. Also we define $A$ as \textit{nicely edge distance-balanced} ($NEDB$), whenever there exists a positive integer $\acute{\gamma}_A$, so that for every edge $f=\alpha\beta$ we have: $m_\alpha^A(f)=m_\beta^A(f)=\acute{\gamma}_A$.\\

It is said $A$ as a graph is \textit{strongly edge distance-balanced} ($SEDB$), whenever for each $i\geq 1$ it holds:
\begin{center}
$\acute{D}_{i-1}^i(f)=\acute{D}_i^{i-1}(f)$
\end{center}
According to the stated definition, clearly if graph be $SEDB$, then it is a $EDB$ graph.\\
In this segment, we would express some examples and results concerning $Gt-EDB$ graphs.

%%%%%%%%%%%%%%%%%%%%%%%%%%%
\section{Main results}
At first, we present an example of $Gt-EDB$ graph.
\begin{example}
The complete bipartite graphs $K_{n,tn}$ are a class of $Gt-EDB$ graphs.
\end{example}
\begin{proof}
Let $X$ and $Y$ be two separate parts of $K_{n,tn}$, such that for every edge $f=\alpha\beta$, $\alpha\in X$ and $\beta\in Y$. As we know $K_{n,tn}$ is a bipartite graph with diameter 2. Then $\alpha$ has $(tn-1)$ adjacent edges and $\beta$ includes $(n-1)$ adjacent edges. According to the definition of $m_\alpha^A(f)$, we obtain $m_\alpha^A(f)=tn$
and $m_\beta^A(f)=n$. Thus, $m_\alpha^A(f)=tm_\beta^A(f)$. This implies that $K_{n,tn}$ is a $Gt-EDB$ graph.
\end{proof}
 
\begin{proposition}
Assume that $A$ is a $Gt-EDB$ bipartite graph, where its diameter is 2. Then for each pair of adjacent vertices $\alpha\beta$, $deg(\alpha)=t deg(\beta)$.
\end{proposition} 
 
\begin{proof}
 Since $A$ is a $Gt-EDB$ graph, then $m_\alpha^A(f)=tm_\beta^A(f)$ and we have 
 \begin{center}
$|\{f\}\bigcup^{d-1}_{i=1}\acute{D}_{i-1}^i(f)|=t|\{f\}\bigcup^{d-1}_{i=1}\acute{D}_i^{i-1}(f)|$.
\end{center} 
 Therefore,\\
 $\hspace*{2.5cm}$ $\sum^{d-1}_{i=1}|\acute{D}_{i-1}^i(f)|=t\sum^{d-1}_{i=1}|\acute{D}_i^{i-1}(f)|+(t-1)$.$\hspace*{1.9cm}(1)$\\\\
If $|\acute{D}^2_1(f)|=r$, then $|\acute{D}^1_2(f)|=tr+(t-1)$. Thus, $deg(\beta)=r+1$ and $deg(\alpha)=tr+t=t(r+1)$. Therefore, always $deg(\alpha)=t deg(\beta)$.
 \end{proof}
 
 \begin{lemma}
 Suppose that $A$ is a $Gt-EDB$ bipartite graph with $d=2$. Then $A$ can only be $K_{n,tn}$.
 \end{lemma}
 
 \begin{proof}
Consider that $A$ is a $Gt-EDB$ bipartite graph, in which its diameter is equal to 2. It is asserted that $A$ is a complete bipartite graph. If not, it can not have a diameter of 2. The result is obtained from Proposition 2.2 that, $deg(\alpha)=t deg(\beta)$. If $A$ be a complete bipartite graph, then $A$ must be $K_{n,tn}$.
 \end{proof}
 
 \begin{proposition}
 A bipartite graph $A$ is $Gt-EDB$ if and only if
 \begin{center}
 $S_{{z}_{e}}(A)=t.\frac{(tn+t+1)^2.tn^2}{(t+1)^2}$.
 \end{center}
 \end{proposition}
 \begin{proof}
 Persume that $A$ is $Gt-EDB$. For every edge $f=\alpha\beta$, we have $m_\alpha^A(f)=tm_\beta^A(f)$ and since $A$ is bipartite also $m_\alpha^A(f)+m_\beta^A(f)+m_0^A(f)=|E(A)|-1$ holds. Therefore,
 \begin{center}
$m_\alpha^A(f)+m_\beta^A(f)=|E(A)|-m_0^A(f)-1$.
\end{center}
Then,
\begin{center}
$t^2(m_\beta^A(f))^2+(m_\beta^A(f))^2+2t(m_\beta^A(f))^2=(|E(A)|-m_0^A(f)-1)^2$.
\end{center}
Hence,
\begin{center}
$(m_\beta^A(f))^2=\frac{(|E(A)|-m_0^A(f)-1)^2}{(t+1)^2}$,
\end{center}
and so,
 \begin{center}
 $\sum_{f\in E(A)}m_{\alpha}^A(f).m_{\beta}^A(f)=t\sum_{f\in E(A)}(m_{\beta}^A(f))^2=$
 \end{center}
 \begin{center}
 $t.\frac{(|E(A)|-m_0^A(f)-1)^2.|E(A)|}{(t+1)^2}=t.\frac{(tn+t+1)^2.tn^2}{(t+1)^2}$.
 \end{center}
For converse, let $S_{{z}_{e}}(A)=t.\frac{(tn+t+1)^2.tn^2}{(t+1)^2}$. Since $A$ is a graph, we have 
$m_\alpha^A(f)+m_\beta^A(f)+m_0^A(f)=|E(A)|-1$ and also 
\begin{center}
$m_0^A(f)+(t+1)m_\beta^A(f)=|E(A)|-1$.
\end{center}
Hence,
\begin{center}
$m_\beta^A(f)=\frac{(|E(A)|-m_0^A(f)-1)}{(t+1)}$.
\end{center}
As well as 
\begin{center}
$m_\alpha^A(f).m_\beta^A(f)=t.\frac{(|E(A)|-m_0^A(f)-1)^2}{(t+1)^2}$,
\end{center}
and hence
\begin{center}
$m_\alpha^A(f)=tm_\beta^A(f)=t.\frac{(|E(A)|-m_0^A(f)-1)}{(t+1)}$.
\end{center}
This implies that $A$ is $Gt-EDB$.
 \end{proof}

\section{$Gt-EDB$ graphs and graph products}
 
We would now investigate situations in which the \textit{Cartesian product} leads to a $Gt-EDBG$. We mention that such product graphs, formed by graphs $A$ and $B$, its vertex set is $V(A\square B)=V(A)\times V(B)$. Consider that $(a_1,b_1)$ and $(a_2,b_2)$ are detached vertices in $V(A\square B)$. In the Cartesian product $A\square B$, if vertices $(a_1,b_1)$ and $(a_2,b_2)$ are coincident in one coordinate and adjacent in the another coordinate, then they are adjacent, that is, $a_1=a_2$ and $b_1b_2\in E(B)$, or $b_1=b_2$ and $a_1a_2\in E(A)$. Obviously, for vertices we have:
\begin{center}
$d_{A\square B}((a_1,b_1),(a_2,b_2))=d_A(a_1,a_2)+d_B(b_1,b_2)$.
\end{center}
For edges we have:
\begin{center}
$d_{A\square B}((a,b)(a_1,b_1),(\acute{a},\acute{b})(\acute{a}_1,\acute{b}_1))=$
\end{center}
\begin{center}
$\min\{d_{A\square B}((a,b),(\acute{a},\acute{b})),d_{A\square B}((a,b),(\acute{a}_1,\acute{b}_1)),d_{A\square B}((a_1,b_1),(\acute{a},\acute{b}))$, \\ $d_{A\square B}((a_1,b_1),(\acute{a}_1,\acute{b}_1))\}=$
 \end{center}
 \begin{center}
$\min\{d_A(a,\acute{a})+d_B(b,\acute{b}),d_A(a,\acute{a}_1)+d_B(b,\acute{b}_1),
d_A(a_1,\acute{a})+d_B(b_1,\acute{b}),d_A(a_1,\acute{a}_1)+d_B(b_1,\acute{b}_1)\}$.
\end{center}

\begin{proposition}
Suppose that $A$ and $B$ are connected graphs. Then $A\square B$ is a $Gt-EDB$ graph if and only if $A$ and $B$ are both $Gt-EDB$ and $Gt-DB$ graphs.
\end{proposition}
\begin{proof}
Consider that $a_1,a_2$ are adjacent vertices in $V(A)$, and $b_1,b_2$ are two adjacent vertices in $V(B)$. We assume that $(a_1,b_1) ,(a_2,b_1) ,(a_1,b_2)\in V(A\square B)$. We observe that
\begin{center}
$(a,b)(\acute{a},\acute{b})\in M_{(a_1,b_1)}((a_1,b_1)(a_2,b_1))\Leftrightarrow$
\end{center}
\begin{center}
$\min\{d_{A\square B}((a,b),(a_1,b_1)),d_{A\square B}((\acute{a},\acute{b}),(a_1,b_1))\}<\min\{d_{A\square B}((a,b),(a_2,b_1)),d_{A\square B}((\acute{a},\acute{b}),(a_2,b_1))\}\Leftrightarrow$
\end{center}
\begin{center}
$\min \{d_A(a,a_1)+d_B(b,b_1),d_A(\acute{a},a_1)+d_B(\acute{b},b_1)\}<\min \{d_A(a,a_2)+d_B(b,b_1),d_A(\acute{a},a_2)+d_B(\acute{b},b_1)\}\Leftrightarrow$
\end{center}
\begin{center}
$\min\{d_A(a,a_1),d_A(\acute{a},a_1)\}<\min\{d_A(a,a_2),d_A(\acute{a},a_2)\}$.
\end{center}
We deduce that
\begin{center}
$M_{(a_1,b_1)}((a_1,b_1)(a_2,b_2))=\{(a,b)(\acute{a},\acute{b})\in E(A\square B)|a\acute{a}\in E(A) , b=\acute{b}$ or\\ $b\acute{b}\in E(B)$ , $a=\acute{a}$ ,
\end{center}

$\hspace*{1cm}$ $\min\{d_A(a,a_1),d_A(\acute{a},a_1)\}<\min\{d_A(a,a_2),d_A(\acute{a},a_2)\}\}$, $\hspace*{1.2cm}$(2)\\\\
and hence applying (2) and by [19. Theorem 2.1] we conclude that:\\\\
$\hspace*{1.1cm}$ $m^{A\square B}_{(a_1,b_1)}((a_1,b_1)(a_2,b_2))=m^A_{a_1}(a_1a_2).|V(B)|+n^A_{a_1}(a_1a_2).|E(B)|$.$\hspace*{0.25cm}$(3)\\\\
Similar to this process we obtain\\\\
$\hspace*{1.1cm}$ $m^{A\square B}_{(a_2,b_2)}((a_2,b_2)(a_1,b_1))=m^A_{a_2}(a_1a_2).|V(B)|+n^A_{a_2}(a_1a_2).|E(B)|$,$\hspace*{0.27cm}$(4)\\\\
$\hspace*{1cm}$ $m^{A\square B}_{(a_1,b_1)}((a_1,b_1)(a_1,b_2))=m^B_{b_1}(b_1b_2).|V(A)|+n^B_{b_1}(b_1b_2).|E(A)|$,$\hspace*{0.65cm}$(5)\\\\
$\hspace*{1cm}$ $m^{A\square B}_{(a_1,b_2)}((a_1,b_2)(a_1,b_1))=m^B_{b_2}(b_1b_2).|V(A)|+n^B_{b_2}(b_1b_2).|E(A)|$.$\hspace*{0.65cm}$(6)\\\\
Persume that both $A$ and $B$ are both $Gt-EDBG$ and $Gt-DBG$. Since $A$ is $Gt-EDB$ and $Gt-DB$, then by Equalities (3) and (4) we have:\\\\
$m^A_{a_1}(a_1a_2).|V(B)|+n^A_{a_1}(a_1a_2).|E(B)|=t m^A_{a_2}(a_1a_2).|V(B)|+t n^A_{a_2}(a_1a_2).|E(B)|$,\\\\
Therefore,
\begin{center}
$m^{A\square B}_{(a_1,b_1)}((a_1,b_1)(a_2,b_2))=t m^{A\square B}_{(a_2,b_2)}((a_1,b_1)(a_2,b_2))$.
\end{center}
By analogy, using (5) and (6), we conclude that
\begin{center}
$m^{A\square B}_{(a_1,b_1)}((a_1,b_1)(a_1,b_2))=t m^{A\square B}_{(a_1,b_2)}((a_1,b_1)(a_1,b_2))$,
\end{center}
and hence $A\square B$ is $Gt-EDB$.\\
For converse, let $A\square B$ is $Gt-EDB$, then applying (3) and (4) we observe that
\begin{center}
$m^{A\square B}_{(a_1,b_1)}((a_1,b_1)(a_2,b_2))=t m^{A\square B}_{(a_2,b_2)}((a_1,b_1)(a_2,b_2))\Rightarrow$
\end{center}
$m^A_{a_1}(a_1a_2).|V(B)|+n^A_{a_1}(a_1a_2).|E(B)|=t m^A_{a_2}(a_1a_2).|V(B)|+t n^A_{a_2}(a_1a_2).|E(B)|$,\\\\
Hence $A$ is $Gt-EDB$ and $Gt-DB$. In the same way, (5) and (6) yield that $B$ is $Gt-EDB$ and $Gt-DB$. This completes the result.
\end{proof}

We would define the lexicographic product graphs. Product graph $A[B]$ of the graphs $A$ and $B$, where its vertex set is $V(A[B])=V(A)\times V(B)$, and two its adjacent vertices are $(a_1,b_1) , (a_2,b_2)$ is defined the \emph{lexicographic product} if $a_1a_2\in E(A)$ or if $a_1=a_2$ and also $b_1b_2\in E(B)$ (for more information see [12,p.22]). Since $A$ is a graph, thus it is easily seen for vertices that
\begin{center}
$\ d_{A[B]}((a_1,b_1),(a_2,b_2))=\left\{ \begin{array} {cc} d_A(a_1,a_2) & if a_1\neq a_2\\ \min\{2,d_B(b_1,b_2)\} & if a_1=a_2. \end{array}\right.$
\end{center}

And for edges we have:
\begin{center}
$ d_{A[B]}((a,b)(a_1,b_1),(\acute{a},\acute{b})(\acute{a}_1,\acute{b}_1))=$
\end{center}

\begin{equation}\min 
\left\{
\begin{array}{cccc}
d_A(a,\acute{a}) & if a\neq \acute{a}, & \min\{2,d_B(b,\acute{b})\} & if a=\acute{a} \notag\\
d_A(a,\acute{a_1}) & if a\neq \acute{a_1}, & \min\{2,d_B(b,\acute{b_1})\} & if a=\acute{a_1}\\
d_A(a_1,\acute{a}) & if a_1\neq \acute{a}, & \min\{2,d_B(b_1,\acute{b})\} & if a_1=\acute{a}\\
d_A(a_1,\acute{a}_1) & if a_1\neq \acute{a}_1, & \min\{2,d_B(b_1,\acute{b}_1)\} & if a_1=\acute{a}_1
\end{array}
\right\}.
\end{equation}

\begin{proposition}
Persume that $A$ and $B$ are graphs. Then there is the lexicographic product $A[B]$, and $A[B]$ is $Gt-EDB$ if and only if $A$ is $Gt-EDB$ and also $B$ is an empty graph.
\end{proposition}
\begin{proof}
Let us choose $A$ as a connected graph, and $A[B]$ as a $Gt-EDBG$. Therefore, it yields $A[B]$ is a bipartite graph based on Lemma 2.3. Since $B$ has at least an edge, then it is understandable that $A[B]$ is not a bipartite graph. Therefore, $B$ is an empty graph. Assume that $a_1$ and $a_2$ be adjacent vertices in $A$, and $b_1,b_2\in V(B)$ ($b_1$  and $b_2$ are not necessarily distinct). We inspire of the definition of the lexicographic product that $(a_1,b_1)$ and $(a_2,b	_2)$ are adjacent, and using (2) we observe that,

\begin{center}
$M^{A[B]}_{(a_1,b_1)}=\{(a_1,b_1)(a,b)\}\cup \{(a_2,\acute{b})(a,b)|\acute{b}\in  V(B)\backslash \{b_2\}\}$
\end{center}
\begin{center}
$\cup \{(a,b)(a_1,\acute{b})| aa_1\in M^A_{a_1}(a_1a_2), a\acute{a}\neq a_1a_2, \acute{b}\in V(B)$\},
\end{center}
where $a\acute{a}\neq a_1a_2$, that is $\{a\neq a_1, \acute{a}\neq a_2\}$ or $\{\acute{a}\neq a_2, \acute{a}\neq a_1\}$.\\
Therefore, $m^{A[B]}_{(a_1,b_1)}((a_1,b_1)(a_2,b_2))=|V(B)| . m^A_{a_1}(a_1a_2)$. Similarly we attain
$m^{A[B]}_{(a_2,b_2)}((a_1,b_1)(a_2,b_2))=|V(B)| . m^A_{a_2}(a_1a_2)$. For an empty graph $B$ it is simple to be realized, the lexicographic product $A[B]$ is $Gt-EDB$ if and only if $A$ is $Gt-EDB$. The proof is completed.
\end{proof}

\section{$Gt-NEDB$ and $Gt-SEDB$ graphs}
For every edge $f=\alpha\beta\in E(A)$ in a $Gt-EDB$ graph, $m_\alpha^A(f)=tm_\beta^A(f)$. Consider that for every arbitrary edge $f\in E(A)$, the constant of $m_\beta^A(f)$ is $\acute{\gamma}_A$. Then $A$ is introduced as \textit{Generalized $t$-nicely edge distance-balanced} ($Gt-NEDB$). Each $Gt-NEDB$ graph is $Gt-EDB$ and every bipartite $Gt-EDB$ graph is $Gt-NEDB$. At first, we start with the below lemma.

\begin{lemma}
If $A$ be a connected $Gt-NEDB$ graph with diameter of $d$, then for each arbitrary edge $f=\alpha\beta\in E(A)$, there will precisely be $|E(A)|-(t+1)\acute{\gamma}_A$ edges of $A$, such that will be at the equal distance from $\alpha$ and $\beta$. On the other hand, $m_0^A(f)=|E(A)|-(t+1)\acute{\gamma}_A-1$.
\end{lemma}
\begin{proof}
It is directly deduced from $m_\alpha^A(f)+m_\beta^A(f)+m_0^A(f)=|E(A)|-1$.
\end{proof}

\begin{lemma}
Suppose that $A$ is a connected $Gt-NEDB$ graph with diameter $d$. Then $d-1\leq t\acute{\gamma}_A$.
\end{lemma}

\begin{proof}
Let a path $a_0,a_1,a_2,...,a_{d-1},a_d$. Also, we have $h_1,h_2,h_3,...,h_d$ as a series of edges, in which $d_A(h_1,h_d)=d$. Now, let $h_1=a_0a_1$ be an arbitrary edge of $A$. without loss of generality, we asuume that $m_{a_1}^A(h_1)=tm_{a_0}^A(h_1)$. Hence, $|\{h_2,...,h_d\}|=d-1\leq m_{a_1}^A(h_1)=tm_{a_0}^A(h_1)$. This shows that  $d-1\leq t\acute{\gamma}_A$.
\end{proof}

\begin{lemma}
Persume that $A$ and $B$ are graphs. Then, $A\square B$ is $Gt-NEDB$ if and only if both $A$ and $B$ are $Gt-NEDB$ and 
\begin{center}
$|V(B)|.\acute{\gamma}_A+|E(B)|.\gamma_A=|V(A)|.\acute{\gamma}_B+|E(A)|.\gamma_B$.
\end{center}
\end{lemma}

\begin{proof}
Consider that vertices $(a_1,b_1)$ and $(a_2,b_2)$ are in $V(A\square B)$. According to the definition, either $b_1=b_2$ and $a_1$ and $a_2$ are adjacent in $A$, or $a_1=a_2$ and $b_1$ and $b_2$ are adjacent in $B$. Let now $b_1=b_2$ and $a_1$ and $a_2$ be adjacent in $A$. Assume now that $A\square B$ is $Gt-NEDB$. Using equations (3),(4),(5) and (6), we obtain
\begin{center}
$m_{a_1}^A(a_1a_2).|V(B)|+n_{a_1}^A(a_1a_2).|E(B)|=t\{m_{a_2}^A(a_1a_2).|V(B)|+n_{a_2}^A(a_1a_2).|E(B)|\}$=
$m_{b_1}^B(b_1b_2).|V(A)|+n_{b_1}^B(b_1b_2).|E(A)|=t\{m_{b_2}^B(b_1b_2).|V(A)|+n_{b_2}^B(b_1b_2).|E(A)|\}$=
\end{center}

$\hspace*{4cm}$ $t(\acute{\gamma}_{A\square B}+\gamma_{A\square B})$,$\hspace*{4.6cm}(7)$

in which
\begin{center}
$\gamma_{A\square B}=n_{a_1}^A(a_1a_2).|E(B)|=n_{a_2}^A(a_1a_2).|E(B)|$,
\end{center}
and
\begin{center}
$\acute{\gamma}_{A\square B}=m_{a_1}^A(a_1a_2).|V(B)|=m_{a_2}^A(a_1a_2).|V(B)|$.
\end{center}
We conclude that from (7) that $A$ and $B$ are both $Gt-NEDB$ and $Gt-NDB$ and
$|V(B)|.\acute{\gamma}_A+|E(B)|.\gamma_A=|V(A)|.\acute{\gamma}_B+|E(A)|.\gamma_B$ holds.
Conversely, if $A$ and $B$ are both $Gt-NEDB$ and $Gt-NDB$ with $|V(B)|.\acute{\gamma}_A+|E(B)|.\gamma_A=|V(A)|.\acute{\gamma}_B+|E(A)|.\gamma_B$. By equations (3),(4),(5) and (6), $A\square B$ is a $Gt-NEDB$ graph.
\end{proof}

While for each arbitrary edge $f=\alpha\beta \in E(A)$ and every $i\in[1,d-1]$ in a graph $A$ with diameter $d$, we have $|\acute{D}_{i-1}^i(f)|=t|\acute{D}_i^{i-1}(f)|+(t-1)$, then such graphs are called $Gt-SEDB$. Let $t=1$. Then graph $A$ is a $SEDB$ graph.
\begin{example}
A class of $Gt-SEDB$ graphs are complete bipartite graphs $K_{n,tn}$.
\end{example}

\begin{proof}
Let $f=\alpha\beta\in E(K_{n,tn})$. Since $K_{n,tn}$ is bipartite with diameter 2, we attain $\acute{D}_1^2(f)=(n-1)$ and $\acute{D}_2^1(f)=t(n-1)$. Therefore, $\acute{D}_2^1(f)=t\acute{D}_1^2(f)+(t-1)$ and it completes the result.
\end{proof}

\begin{lemma}
Consider that $A$ is a $Gt-SEDB$ graph with diameter 2. Then $A$ is a $Gt-EDB$ graph for $t\in\mathbb{N}$ and $t>1$.
\end{lemma}

\begin{proof}
According to (1), we know that the graph $A$ is $Gt-EDB$ if and only if for each arbitrary edge $f=\alpha\beta\in E(A)$ and each $i\in[1,d-1]$, we have
\begin{center}
$\sum_{i=1}^{d-1}|\acute{D}_i^{i-1}(f)|=t\sum_{i=1}^{d-1}|\acute{D}_{i-1}^i(f)|+(t-1)$.
\end{center} 
Let now $A$ be a $Gt-SEDB$ graph. Then for each $f=\alpha\beta\in E(A)$ and each $i\in[1,d-1]$, it is obtained
\begin{center}
$|\acute{D}_i^{i-1}(f)|=t|\acute{D}_{i-1}^i(f)|+(t-1)$.
\end{center}
Thus, 
\begin{center}
$\sum_{i=1}^{d-1}|\acute{D}_i^{i-1}(f)|=t\sum_{i=1}^{d-1}|\acute{D}_{i-1}^i(f)|+(t-1)(d-1)$.
\end{center}
If $d=2$, then we obtain
\begin{center}
$|\acute{D}_2^1(f)|=t|\acute{D}_1^2(f)|+(t-1)$.
\end{center}
The proof is completed.
\end{proof}
Pay attention that If $t=1$, then every $SEDB$ graph is $EDB$.

%%%%%%%%%%%%%%%%%%%%%%%%%%%%%%%%%%%%%%%%%%%
% References
%%%%%%%%%%%%%%%%%%%%%%%%%%%%%%%%%%%%%%%%%%%
\bibliographystyle{amsplain}

\begin{thebibliography}{5}

\bibitem{Ha} A.Abedi, M. Alaeiyan, A. Hujdurovi\'{c}, K. Kutnar, Quasi-$\lambda$-distance-balanced
graphs, {\em Discrete Appl. Math.} {\bf 227} (2017), 21-28.
\bibitem{Ha} Z. Aliannejadi, A. Gilani, M. Alaeiyan, J. Asadpour, On some properties of edge quasi-distance-balanced graphs, {\em Journal of Mathematical Extension.} {\bf 16} (2022), 1-13.
\bibitem{Ha} Z. Aliannejadi, M. Alaeiyan, A. Gilani, Strongly edge distance-balanced graph products, {\em 7th International Conference on Combinatorics, Cryptography, Computer Science and Computing} (2022).
\bibitem{Ha} K. Balakrishman, M. Changat, I. Peterin, S.P. \v{S}pacapan, P. \v{S}paral, A. R. Subhamathi, Strongly
distance-balanced graph and graph product, {\em European. J. Combin.} {\bf 30} (2009),
1048-1053.
\bibitem{Ha} D.M. Cvetkoci\'{c}, M. Doob, H. Sachs, Spectra of Graphs-Theory and Application, {\em Academic Press, New York.} (1980).
\bibitem{Ha}
 M. Faghani, A.R. Ashrafi, Revised and edge revised Szeged indices of graphs, {\em Ars Math. Contemp.} {\bf 7} (2014), 153-160. 
\bibitem{Ha}
A. Graovac, M. Juvan, M. Petkovsek, A. Vesel, J. Zerovnik, The Szeged index of fasciagraphs, {\em Match Common. Chem.} {\bf 49} (2003), 47-66. 
\bibitem{Ha}
I. Gutman, L. Popovic, P.V. Khadikar, S. Karmarkar, S. Joshi, M. Mandloi, Relations between Wiener and Szeged indices of monocyclic molecules, {\em Match Common. Math. Comput. Chem.} {\bf 35} (1997), 91-103. 
\bibitem{Ha}
I. Gutman, A.R. Ashrafi, The edge version of the Szeged index, {\em Croat. Chem. Acta.} {\bf 81} (2008), 263-266. 
\bibitem{Ha}
A. Hujdurovi\'{c}, On some properties of quasi-distance-balanced graphs, {\em Bull. Aust. Math. Soc.} {\bf 97} (2018), 177-184. 
\bibitem{Ha} A. Ili\v{c}, S. Klav\v{z}ar, M. Milanovi\'{c}, On distance-balanced graphs, {\em European. J.
Combin.} {\bf }31 (2010), 733-737.
\bibitem{Ha}
W. Imrich, S. Klav\v{z}ar, Product Graphs: Structure and Recognition, {\em Wiley, New York, USA}, (2000).
\bibitem{Ha}
J. Jerebic, S. Klav\v{z}ar, D.F. Rall, Distance-balanced graphs, {\em Ann. Comb.} {\bf 12} (2008), 71-79. 
\bibitem{Ha}
M.H. Khalifeh, H. Yousefi-Azari, A.R. Ashrafi, S.G. Wagner, Some new results on distance-based graph invariants, {\em European J. Combin.} {\bf 30} (2009), 1149-1163. 
\bibitem{Ha} K. Kutnar, A. Malni\v{c}, D. Maru\v{s}i\v{c}, \v{S}. Miklavi\v{c}, Distance-balanced
graphs:symmetry conditions,{\em Discrete Math.} {\bf 306} (2006), 1881-1894.
\bibitem{Ha} K. Kutnar, A. Malni\v{c}, D. Maru\v{s}i\v{c}, \v{S}. Miklavi\v{c}, The strongley distance-balanced
property of generalized petersen graphs, {\em Ars Math. Contemp.} {\bf 2} (2009), 41-47.
\bibitem{Ha} K. Kutnar, \v{S}. Miklavi\v{c}, Nicely distance-balanced graphs,{\em European. j. Combin.}
{\bf 39} (2014), 57-67.
\bibitem{Ha} \v{S}. Miklavi\v{c}, P. \v{S}parl, On the connectivity of bipartite distance-balanced graphs,
{\em European. J. Combin.} {\bf 33} (2012), 237-247.
\bibitem{Ha}
M. Tavakoli, H. Yousefi-azari, A.R. Ashrafi, Note on Edge Distance-Balanced Graphs, {\em Transaction on Combinatorics, University of Isfehan} {\bf 1} no. 1 (2012), 1-6. 


\end{thebibliography}
%%%%%%%%%%%%%%%%%%%%%%%%%%%%%%%%%%%%%%%%%%%
% Please cite your relevant papers but at most total 5 papers/books.
%%%%%%%%%%%%%%%%%%%%%%%%%%%%%%%%%%%%%%%%%%%

\bigskip
\bigskip

%{\bf Received: Month xx, 200x}

{\footnotesize {\bf Zohreh Aliannejadi}\; \\ {Department of mathematics, Islamic Azad university, South-Branch,
Tehran, Iran. }\\
{\tt Email:alian.zohreh64@gmail.com }\\

{\footnotesize {\bf Mehdi Alaeiyan}\; \\ {Department of Mathematics, Department of mathematics, Iran university of science and technology, Narmak,
Tehran 16844. Iran. }\\
{\tt Email: alaeiyan@iust.ac.ir }\\

{\footnotesize {\bf Alireza Gilani}\; \\ {Department of mathematics, Islamic Azad university, South-Branch,
Tehran, Iran. }\\
{\tt Email: a\_ gilani@azad.ac.ir }\\

\end{document}